\documentclass{amsart}
\usepackage{amssymb}
\usepackage{amsfonts}

\newtheorem{theorem}{Theorem}
\theoremstyle{plain}

\newtheorem{lemma}{Lemma}

\numberwithin{equation}{section}
\input{tcilatex}

\begin{document}
\title{On some classes of Abel-Grassmann's groupoids}
\subjclass[2000]{20M10 and 20N99}
\author[M. Khan, Faisal, V. Amjid]{}
\maketitle

\begin{center}
$^{1}$\textbf{Madad Khan, }$^{2}$\textbf{Faisal and }$^{3}$\textbf{Venus
Amjid}

\bigskip

\textbf{Department of Mathematics}

\textbf{COMSATS Institute of Information Technology}

\textbf{Abbottabad, Pakistan}

\bigskip

$^{1}$\textbf{madadmath@yahoo.com, }$^{2}$\textbf{yousafzaimath@yahoo.com}

$^{3}$\textbf{venusmath@yahoo.com}
\end{center}

\bigskip

\textbf{Abstract.} In this paper, we have investigated different classes of
an AG-groupoid by their structural properties. We have prove that weakly
regular, intra-regular, right regular, left regular, left quasi regular and
completely regular coincide in an AG-groupoid with left identity and in AG$%
^{\ast \ast }$-groupoid. Further we have prove that every regular (weakly
regular, intra-regular, right regular, left regular, left quasi regular,
completely regular) AG-groupoid with left identity $($AG$^{\ast \ast }$%
-groupoid$)$ is regular but the converse is not true in general. Also it has
been shown that non-associative regular, weakly regular, intra-regular,
right regular, left regular, left quasi regular and completely regular AG$%
^{\ast }$ groupoids do not exist.

\bigskip

\textbf{Keywords}. An AG-groupoid, left invertive law, medial law and
paramedial law.

\begin{center}
\bigskip

{\LARGE Introduction}
\end{center}

\bigskip

Abel-Grassmann's groupoid (AG-groupoid) \cite{ref10}, is a groupoid $S$
holding the left invertive law for\ all $a,b,c\in S$%
\begin{equation}
(ab)c=(cb)a\text{.}  \tag{$1$}
\end{equation}%
In an AG-groupoid the medial law holds for\ all $a,b,c,d\in S$%
\begin{equation}
(ab)(cd)=(ac)(bd)\text{.}  \tag{$2$}
\end{equation}%
An AG-groupoid is a non-associative algebraic structure mid way between a
groupoid and a commutative semigroup with wide applications in theory of
flocks \cite{Naseeruddin}. This structure is closely related with a
commutative semigroup because if an AG-groupoid contains a right identity,
then it becomes a commutative semigroup \cite{Mus3}. There can be a unique
left identity in an AG-groupoid \cite{Mus3}.

In an AG-groupoid $S$ with left identity, the paramedial laws hold for\ all $%
a,b,c,d\in S$%
\begin{equation}
(ab)(cd)=(dc)(ba).  \tag{$3$}
\end{equation}%
Further if an AG-groupoid contain a left identity, the following law holds
for all $a,b,c\in S$%
\begin{equation}
a(bc)=b\left( ac\right) \text{.}  \tag{$4$}
\end{equation}

If an AG-groupoid satisfies $(4)$, then it becomes an AG$^{\ast \ast }$%
-groupoid. Thus in an AG$^{\ast \ast }$-groupoid$,$ $(3)$ and $(4)$ holds
without left identity.

An element $a$ of an AG-groupoid $S$ is called a regular if there exists $%
x\in S$ such that $a=(ax)a$ and $S$ is called regular if all elements of $S$
are regular.

An element $a$ of an AG-groupoid $S$ is called a weakly regular if there
exists $x,y\in S$ such that $a=(ax)(ay)$ and $S$ is called weakly regular if
all elements of $S$ are weakly regular.

An element $a$ of an AG-groupoid $S$ is called an intra-regular if there
exists $x,y\in S$ such that $a=(xa^{2})y$ and $S$ is called an intra-regular
if all elements of $S$ are intra-regular (see \cite{n}).

An element $a$ of an AG-groupoid $S$ is called a right regular if there
exists $x\in S$ such that $a=a^{2}x=(aa)x$ and $S$ is called a right regular
if all elements of $S$ are right regular.

An element $a$ of an AG-groupoid $S$ is called a left regular if there
exists $x\in S$ such that $a=xa^{2}=x(aa)$ and $S$ is called left regular if
all elements of $S$ are left regular.

An element $a$ of an AG-groupoid $S$ is called a left quasi regular if there
exists $x,y\in S$ such that $a=(xa)(ya)$ and $S$ is called left quasi
regular if all elements of $S$ are left quasi regular.

An element $a$ of an AG-groupoid $S$ is called a completely regular if $a$
is regular, left regular and right regular. $S$ is called completely regular
if it is regular, left and right regular.

\begin{lemma}
\label{ss}If $S$ is regular $($weakly regular, intra-regular, right regular,
left regular, left quasi regular, completely regular$)$ AG groupoid$,$ then $%
S=S^{2}.$
\end{lemma}

\begin{proof}
Let $S$ be a regular AG-groupoid, then $S^{2}\subseteq S$ is obvious. Let $%
a\in S,$ then since $S$ is regular so there exists $x\in S$ such that $%
a=(ax)a.$ Now%
\begin{equation*}
a=(ax)a\in SS=S^{2}.
\end{equation*}

Similarly if $S$ is weakly regular, intra-regular, right regular, left
regular, left quasi regular and completely regular, then we can show that $%
S=S^{2}$.
\end{proof}

The converse is not true in general which can be followed from the following
example.

Let us consider an AG-groupoid $S=\left \{ a,b,c,d,e,f\right \} $ in the
following Caylay's table.

\begin{center}
\begin{tabular}{c|cccccc}
. & $a$ & $b$ & $c$ & $d$ & $e$ & $f$ \\ \hline
$a$ & $a$ & $a$ & $a$ & $a$ & $a$ & $a$ \\ 
$b$ & $a$ & $b$ & $b$ & $b$ & $b$ & $b$ \\ 
$c$ & $a$ & $b$ & $f$ & $f$ & $d$ & $f$ \\ 
$d$ & $a$ & $b$ & $f$ & $f$ & $c$ & $f$ \\ 
$e$ & $a$ & $b$ & $c$ & $d$ & $e$ & $f$ \\ 
$f$ & $a$ & $b$ & $f$ & $f$ & $f$ & $f$%
\end{tabular}
\end{center}

Note that $S=S^{2}$ but $S$ is not a regular (weakly regular, intra-regular,
right regular, left regular, left quasi regular, completely regular) because 
$d\in S$ is not regular, weakly regular, intra-regular, right regular, left
regular, left quasi regular and completely regular.

\begin{theorem}
If $S$ is an AG-groupoid with left identity $($AG$^{\ast \ast }$-groupoid$),$
then $S$ is an intra-regular if and only if for all $a\in S,$ $a=(xa)(az)$
holds for some $x,z\in S.$
\end{theorem}

\begin{proof}
Let $S$ be an intra-regular AG-groupoid with left identity $($AG$^{\ast \ast
}$-groupoid$)$, then for any $a\in S$ there exists $x,y\in S$ such that $%
a=(xa^{2})y.$ Now by using Lemma \ref{ss}, $y=uv$ for some $u,v\in S$. Let $%
vu=t$, $yt=s,$ $xs=w$ and $wa=z$ for some $t,s,w,z\in S.$ Thus by using $(4)$%
, $(1)$ and $(3)$, we have%
\begin{eqnarray*}
a &=&(xa^{2})y=(x(aa))y=(a(xa))y=(y(xa))a=(y(xa))((xa^{2})y) \\
&=&((uv)(xa))((xa^{2})y)=((ax)(vu))((xa^{2})y)=((ax)t)((xa^{2})y) \\
&=&(((xa^{2})y)t)(ax)=((ty)(xa^{2}))(ax)=((a^{2}x)(yt))(ax) \\
&=&((a^{2}x)s)(ax)=((sx)(aa))(ax)=((aa)(xs))(ax) \\
&=&((aa)w)(ax)=((wa)a)(ax)=(za)(ax)=(xa)(az).
\end{eqnarray*}

Conversely, let for all $a\in S,$ $a=(xa)(az)$ holds for some $x,z\in S.$
Now by using $(4)$, $(1)$, $(2)$ and $(3)$, we have%
\begin{eqnarray*}
a &=&(xa)(az)=a((xa)z)=((xa)(az))((xa)z)=(a((xa)z))((xa)z) \\
&=&(((xa)z)((xa)z))a=(((xa)(xa))(zz))a=(((ax)(ax))(zz))a \\
&=&((a((ax)x))(zz))a=(((zz)((ax)x))a)a=((z^{2}((ax)x))a)a \\
&=&(((ax)(z^{2}x))a)a=((((z^{2}x)x)a)a)a=(((x^{2}z^{2})a)a)a \\
&=&(a^{2}(x^{2}z^{2}))a=(a(x^{2}z^{2}))(aa)=(at)(aa).
\end{eqnarray*}

Where $x^{2}z^{2}=t$ for some $t\in S.$ Now using $(3)$ and $(1)$, we have%
\begin{eqnarray*}
a &=&(at)(aa)=(((at)(aa))t)(aa)=(((aa)(ta))t)(aa) \\
&=&((a^{2}(ta))t)(aa)=((t(ta))a^{2})(aa)=(ua^{2})v,
\end{eqnarray*}

where $t(ta)=u$ and $aa=v$ for some $u,v\in S$. Thus $S$ is intra-regular.
\end{proof}

\begin{theorem}
\label{lop}If $S$ is an AG-groupoid with left identity $($AG$^{\ast \ast }$%
-groupoid$),$ then the following are equivalent.

$(i)$ $S$ is weakly regular.

$(ii)$ $S$ is intra-regular.
\end{theorem}

\begin{proof}
$(i)\Longrightarrow (ii)$ Let $S$ be a weakly regular AG-groupoid with left
identity $($AG$^{\ast \ast }$-groupoid$)$, then for any $a\in S$ there
exists $x,y\in S$ such that $a=(ax)(ay)$ and by Lemma \ref{ss}, $x=uv$ for
some $u,v\in S$. Let $vu=t\in S$, then by using $(3),$ $(1),$ $(4)$ and $%
(2), $ we have%
\begin{eqnarray*}
a &=&(ax)(ay)=(ya)(xa)=(ya)((uv)a)=(ya)((av)u) \\
&=&(av)((ya)u)=(a(ya))(vu)=(a(ya))t \\
&=&(y(aa))t=(ya^{2})t.
\end{eqnarray*}%
Thus $S$ is an intra-regular.

$(ii)\Longrightarrow (i)$ Let $S$ be an intra-regular AG-groupoid with left
identity $($AG$^{\ast \ast }$-groupoid$)$, then for any $a\in S$ there
exists $y,t\in S$ such that $a=(ya^{2})t.$ Now by using Lemma \ref{ss}, $%
(2), $ $(4),$ $(1)$ and $(3),$ we can show that $a=(ax)(ay)$ for some $%
x,y\in S.$ Thus $S$ is weakly regular.
\end{proof}

\begin{lemma}
\label{ki}If $S$ is an AG-groupoid $($AG$^{\ast \ast }$-groupoid$),$ then
the following are equivalent.

$(i)$ $S$ is weakly regular.

$(ii)$ $S$ is right regular.
\end{lemma}

\begin{proof}
$(i)\Longrightarrow (ii)$ Let $S$ be a weakly regular AG-groupoid $($AG$%
^{\ast \ast }$-groupoid$)$, then for any $a\in S$ there exists $x,y\in S$
such that $a=(ax)(ay)$ and let $xy=t$ for some $t\in S$. Now by using $(2),$
we have%
\begin{equation*}
a=(ax)(ay)=(aa)(xy)=a^{2}t.
\end{equation*}

Thus $S$ is right regular.

$(ii)\Longrightarrow (i)$ It follows from Lemma \ref{ss} and $(2)$.
\end{proof}

\begin{lemma}
\label{li}If $S$ is an AG-groupoid with left identity $($AG$^{\ast \ast }$%
-groupoid$),$ then the following are equivalent.

$(i)$ $S$ is weakly regular.

$(ii)$ $S$ is left regular.
\end{lemma}

\begin{proof}
$(i)\Longrightarrow (ii)$ Let $S$ be a weakly regular AG-groupoid with left
identity $($AG$^{\ast \ast }$-groupoid$)$, then for any $a\in S$ there
exists $x,y\in S$ such that $a=(ax)(ay)$ and let $yx=t$ for some $t\in S$.
Now by using $(2)$ and $(3),$ we have%
\begin{equation*}
a=(ax)(ay)=(aa)(xy)=(yx)(aa)=(yx)a^{2}=ta^{2}.
\end{equation*}

Thus $S$ is left regular.

$(ii)\Longrightarrow (i)$ It follows from Lemma \ref{ss}, $(3)$ and $(2).$
\end{proof}

\begin{lemma}
\label{weak}Every weakly regular AG-groupoid with left identity $($AG$^{\ast
\ast }$-groupoid$)$ is regular.
\end{lemma}

\begin{proof}
Assume that $S$ is a weakly regular AG-groupoid with left identity $($AG$%
^{\ast \ast }$-groupoid$)$, then for any $a\in S$ there exists $x,y\in S$
such that $a=(ax)(ay).$ Let $xy=t\in S$ and $t((yx)a)=u\in S.$ Now by using $%
(1)$, $(2),\,(3)$ and $(4)$, we have%
\begin{eqnarray*}
a &=&(ax)(ay)=((ay)x)a=((xy)a)a=(ta)a=(t((ax)(ay)))a \\
&=&(t((aa)(xy)))a=(t((yx)(aa)))a=(t(a((yx)a)))a \\
&=&(a(t((yx)a)))a=(au)a.
\end{eqnarray*}

Thus $S$ is regular.

The converse is not true in general. For this, assume that $S$ is a regular
AG-groupoid $($AG$^{\ast \ast }$-groupoid$)$, then for every $a\in S$ there
exists $x\in S$ such that $a=(ax)a.$ Let us suppose that $S$ is also weakly
regular, then for every $a\in S$ there exists $x,y\in S$ such that $%
a=(ax)(ay).$ Thus%
\begin{equation*}
(ax)a=(ax)(ay).
\end{equation*}

Therefore%
\begin{equation}
a=ay.  \tag{5}
\end{equation}

Now there are two possibilities for $(5)$ to be hold. Either $S$ contains a
right identity or $S$ is idempotent and in both cases $S$ becomes a
commutative semigroup which is contrary to the given fact.
\end{proof}

\begin{theorem}
\label{io}If $S$ is an AG-groupoid with left identity $($AG$^{\ast \ast }$%
-groupoid$),$ then the following are equivalent.

$(i)$ $S$ is weakly regular.

$(ii)$ $S$ is completely regular.
\end{theorem}

\begin{proof}
$(i)\Longrightarrow (ii)$ It follows from Lemmas \ref{ki}, \ref{li} and \ref%
{weak}.

$(ii)\Longrightarrow (i)$ It follows from Lemma \ref{li}.
\end{proof}

\begin{lemma}
\label{lft}If $S$ is an AG-groupoid with left identity $($AG$^{\ast \ast }$%
-groupoid$),$ then the following are equivalent.

$(i)$ $S$ is weakly regular.

$(ii)$ $S$ is left quasi regular.
\end{lemma}

\begin{proof}
It is simple.
\end{proof}

\begin{lemma}
\label{kij}Every intra-regular AG-groupoid with left identity $($AG$^{\ast
\ast }$-groupoid$)$ is regular.
\end{lemma}

\begin{proof}
It can be followed from Theorem \ref{lop} and Lemma \ref{weak}. The converse
is not true in general.
\end{proof}

\begin{lemma}
Every right regular AG-groupoid with left identity $($AG$^{\ast \ast }$%
-groupoid$)$ is regular.
\end{lemma}

\begin{proof}
It can be followed from Lemmas \ref{ki} and \ref{weak}. The converse is not
true in general.
\end{proof}

\begin{lemma}
Every left regular AG-groupoid with left identity $($AG$^{\ast \ast }$%
-groupoid$)$ is regular.
\end{lemma}

\begin{proof}
It can be followed from Lemmas \ref{li} and \ref{weak}. The converse is not
true in general.
\end{proof}

\begin{lemma}
Every completely regular AG-groupoid with left identity $($AG$^{\ast \ast }$%
-groupoid$)$ is regular.
\end{lemma}

\begin{proof}
It can be followed from Theorem \ref{io} and Lemma \ref{weak}. The converse
is not true in general.
\end{proof}

\begin{lemma}
Every left quasi regular AG-groupoid with left identity $($AG$^{\ast \ast }$%
-groupoid$)$ is regular.
\end{lemma}

\begin{proof}
It can be followed from Lemmas \ref{lft} and \ref{weak}. The converse is not
true in general.
\end{proof}

\begin{theorem}
In an AG-groupoid $S$ with left identity $($AG$^{\ast \ast }$-groupoid$),$
the following are equivalent.

$(i)$ $S$ is weakly regular.

$(ii)$ $S$ is an intra-regular.

$(iii)$ $S$ is right regular.

$(iv)$ $S$ is left regular.

$(v)$ $S$ is left quasi regular.

$(vi)$ $S$ is completely regular.
\end{theorem}

\begin{proof}
$(i)\Longrightarrow (ii)$ It follows from Theorem \ref{lop}.

$(ii)\Longrightarrow (iii)$ It follows from Theorem \ref{lop} and Lemma \ref%
{ki}.

$(iii)\Longrightarrow (iv)$ It follows from Lemmas \ref{ki} and \ref{li}.

$(iv)\Longrightarrow (v)$ It follows from Lemmas \ref{li} and \ref{lft}.

$(v)\Longrightarrow (vi)$ It follows from Lemma \ref{lft} and Theorem \ref%
{io}.

$(vi)\Longrightarrow (i)$ It follows from Theorem \ref{io}.
\end{proof}

An AG-groupoid is called an AG$^{\ast }$-groupoid if the following holds for
all $a,b,c\in S$%
\begin{equation*}
(ab)c=b(ac).
\end{equation*}

In an AG$^{\ast }$-groupoid $S$ the following law holds \cite{ref10}%
\begin{equation}
(x_{1}x_{2})(x_{3}x_{4})=(x_{p(1)}x_{p(2)})(x_{P(3)}x_{P(4)})\text{,} 
\tag{6}
\end{equation}%
where $\{p(1),p(2),p(3),p(4)\}$ means any permutation on the set $%
\{1,2,3,4\} $. It is an easy consequence that if $S=S^{2}$, then $S$ becomes
a commutative semigroup.

\begin{theorem}
Regular $($weakly regular, intra-regular, right regular, left regular, left
quasi regular and completely regular$)$ AG$^{\ast }$-groupoids becomes
semigroups.
\end{theorem}

\begin{proof}
It follows from $(6)$ and Lemma \ref{ss}.
\end{proof}

\end{document}